\documentclass[10pt,twoside, a4paper, english, reqno]{amsart}
\usepackage[dvips]{epsfig}
\usepackage{amscd}
\usepackage{a4wide}
\usepackage{amssymb}
\usepackage{amsthm}
\usepackage{amsmath}
\usepackage{bm} 
\usepackage{latexsym}
\usepackage{enumitem}
\usepackage{booktabs}
\usepackage{multirow}
\usepackage{makecell}
\usepackage{tabularx}
\numberwithin{figure}{section}
\numberwithin{table}{section}
\usepackage{graphicx,caption}

\usepackage{calc}
\usepackage{graphicx}
\usepackage{stmaryrd}
\usepackage{algorithm}
\usepackage{algorithmic}

\usepackage{amsmath,amsthm,bm,mathrsfs}
\usepackage[utf8]{inputenc} 
\usepackage{url}            
\usepackage{hyperref}
\hypersetup{
    colorlinks=true,
    linkcolor=blue,
    filecolor=red,      
    urlcolor=blue,
    citecolor=red,
    }
\usepackage{amsfonts}       
\usepackage{nicefrac}       
\usepackage{microtype}      
\theoremstyle{plain}
\usepackage{doi}

\usepackage[foot]{amsaddr}

\usepackage{upref}
\usepackage{hyperref}
\usepackage{ulem}
\usepackage{enumitem}
\newcommand{\norm}[1]{\left\Vert#1\right\Vert}

\newcommand{\R}{\mathbb R}

\newtheorem{theorem}{Theorem}[section]
\newtheorem{proposition}[theorem]{Proposition}
\newtheorem{lemma}[theorem]{Lemma}
\newtheorem{definition}[theorem]{Definition}

\numberwithin{equation}{section}     
\numberwithin{figure}{section}
\numberwithin{table}{section}

\allowdisplaybreaks

\newcounter{asnr}

{\ifnum\value{asnr}=0 \stepcounter{asnr} 
  \begin{enumerate}[label=\textbf{A}.\arabic{enumi}]
    \else
    \begin{enumerate}[label=\textbf{A}.\arabic{enumi},resume] \fi}
{\end{enumerate}}

\newcounter{defnr}

{\ifnum\value{defnr}=0 \stepcounter{defnr} 
  \begin{enumerate}[label=\textbf{D}.\arabic{enumi}]
    \else
    \begin{enumerate}[label=\textbf{D}.\arabic{enumi},resume] \fi}
{\end{enumerate}}

\numberwithin{equation}{section} \allowdisplaybreaks

\title[RK-LDG scheme for the  BO equation]
{Stability of fully Discrete Local Discontinuous Galerkin method for the generalized Benjamin-Ono equation}
\date{}

\author[M. Dwivedi]{Mukul Dwivedi}
\address[Mukul Dwivedi]{\newline
Department of Mathematics, 
	Indian Institute of Technology Jammu,
	Jagti, NH-44 Bypass Road, Post Office Nagrota,
	Jammu - 181221, India}
\email[]{2020rma1031@iitjammu.ac.in}

\author[T. Sarkar]{Tanmay Sarkar}
\address[Tanmay Sarkar]{\newline
	Department of Mathematics, 
	Indian Institute of Technology Jammu,
	Jagti, NH-44 Bypass Road, Post Office Nagrota,
	Jammu - 181221, India}
\email[]{tanmay.sarkar@iitjammu.ac.in}


\thanks{}

\begin{document}

\begin{abstract}
The main purpose of this paper is to design a fully discrete local discontinuous Galerkin (LDG) scheme for the generalized Benjamin-Ono equation. First, we proved the $L^2$-stability for the proposed semi-discrete LDG scheme and obtained a sub-optimal order of convergence for general nonlinear flux. We develop a fully discrete LDG scheme using the Crank-Nicolson (CN) method and fourth-order fourth-stage Runge-Kutta (RK) method in time. Adapting the methodology established for the semi-discrete scheme, we demonstrate the stability of the fully discrete CN-LDG scheme for general nonlinear flux. Additionally, we consider the fourth-order RK-LDG scheme for higher order convergence in time and prove that it is strongly stable under an appropriate time step constraint by establishing a \emph{three-step strong stability} estimate for linear flux. Numerical examples associated with soliton solutions are provided to validate the efficiency and optimal order of accuracy for both methods. 
\end{abstract}
\maketitle
\section{Introduction}\label{sec1}
We consider the following Cauchy problem associated to the Benjamin-Ono equation in the generalized form
\begin{equation}\label{BOeqn}
\begin{cases}
U_t + f(U)_x -\mathcal{H}U_{xx} = 0, \qquad &(x,t) \in \mathbb{R}\times (0,T],\\
U(x,0) = U_0(x), \qquad &x \in \mathbb{R},
\end{cases}
\end{equation}
where $T>0$ is fixed, $U_0$ represents the prescribed initial data, $f$ is the given flux function, and $\mathcal{H}$ denotes the Hilbert transform \cite{dutta2016convergence, thomee1998numerical}, defined by the principle value integral
\begin{align*}
    \mathcal{H}U(x): = \text{P.V.} \frac{1}{\pi} \int_{\R}\frac{U(x-y)}{y}~dy.
\end{align*}

The Benjamin-Ono equation \eqref{BOeqn} is a nonlinear, non-local partial differential equation that finds application in various physical phenomena \cite{ishimori1982solitons}. In particular, the propagation of weakly nonlinear internal long waves in a fluid with a thin region of stratification can be represented by the Benjamin-Ono equation. Originating from the modeling of waves in shallow water, it offers insights into the behavior of these waves, including their propagation and interaction.
Furthermore, we mention that it defines a Hamiltonian system, and with the help of the inverse scattering method (see \cite{fokas1983inverse}), families of localized solitary wave solutions, called soliton solutions \cite{case1979benjamin}, can be obtained under the appropriate assumptions on the initial data. Since the Benjamin-Ono equation is completely integrable, it admits infinitely many conserved quantities \cite{case1979benjamin}. 

The investigation into the well-posedness of the Cauchy problem \eqref{BOeqn} associated to the  Benjamin-Ono equation has been the subject of extensive research over the years. Pioneering work in the local well-posedness was conducted by I{\'o}rio \cite{jose1986cauchy} for the initial data in $H^{s}(\R)$, $s>3/2$, and making use of the conserved quantities, the global well-posedness for data in $H^{s}(\R)$, $s\geq 2$ is demonstrated.
Tao in \cite{tao2004global} also obtained the global well-posedness in $H^1(\R)$ by introducing the gauge transformation. The idea of gauge transformation given in \cite{tao2004global} was further improved by Kenig et al. \cite{ionescu2007global} to carry out the local well-posedness to $H^s(\R)$ for $s\geq0$. Molinet \cite{molinet2008global} has also obtained the global well-posedness for the periodic data in $L^2(\R)$.

It is well-known that due to the effects of dispersion and nonlinear convection, finding a reliable method for the Benjamin-Ono equation is quite a challenging task.
Nevertheless, recent decades have seen the development of various numerical methods to solve the equation \eqref{BOeqn}, with only the pertinent literature referenced here. An implicit finite difference method introduced by Thomée et al. \cite{thomee1998numerical} utilizes the continuous Hilbert transform. More recently, Dutta et al. \cite{dutta2016convergence} demonstrated the convergence of the fully discrete finite difference scheme, which includes the discrete Hilbert transform, and Galtung \cite{galtung2018convergent} devised a convergent Crank-Nicolson Galerkin scheme.

The discontinuous Galerkin (DG) method, a finite element approach, was first introduced by Reed and Hill \cite{reed1973triangular} for the neutron transport equation. It was later extended by Cockburn et al. to tackle nonlinear conservation laws effectively, as detailed in \cite{cockburn2012discontinuous}. However, DG methods face challenges with equations containing higher-order derivatives, which can introduce instability and inconsistency, as noted by Hesthaven \cite{hesthaven2007nodal}.
To address this, Bassi and Rebay \cite{bassi1997high} adapted the DG method, introducing the Runge-Kutta DG (RKDG) variant for compressible Navier-Stokes equations. Cockburn and Shu further generalized this approach in their local discontinuous Galerkin (LDG) method \cite{cockburn1998local}, designed specifically for higher-order problems. LDG transforms equations into first-order systems by introducing auxiliary variables, which approximate lower derivatives. The ``local'' nature of LDG allows these auxiliary variables to be eliminated locally, enabling stable and efficient numerical flux design at interfaces. Thus, LDG addresses higher-order derivatives in a way that ensures stability and accurate solutions.

The LDG method has been developed to deal with equations that have higher-order derivative terms. For instance, Yan and Shu \cite{yan2002local} devised a LDG method for KdV type equations, which have third-order spatial derivatives.  They obtained the error estimates with order of convergence $k+1/2$ in the linear case. Afterwards, Xu and Shu \cite{xu2005local, xu2007error} studied nonlinear convection-diffusion type equations and observed that the LDG method still provides similar levels of accuracy and order of convergence. 
Furthermore, Xu and Shu \cite{xu2010local} expanded the LDG method to handle equations with fourth and fifth order spatial derivatives. Levy et al. \cite{levy2004local} also worked on adapting the LDG method for equations with compactly supported traveling wave solutions appearing in nonlinear dispersive equations.

In more recent times, the LDG method has become popular for dealing with partial differential equations that involve the non-local operator. Xu and Hesthaven \cite{xu2014discontinuous} came up with an LDG method that breaks down the fractional Laplacian of order $\alpha$ $(1<\alpha<2)$ into second-order derivatives and fractional integrals of order $2-\alpha$.  This method turned out to be very effective, giving the optimal rates of $k+1$ in the linear case and $k+1/2$ in the nonlinear setup.
Similarly, Aboelenen \cite{aboelenen2018high} and Dwivedi et al. \cite{dwivedi2024local}  developed an LDG method specifically designed for fractional Schr{\"o}dinger-type equations and fractional Korteweg-de Vries equations, respectively.

Research on developing fully discrete LDG schemes with stability analysis for equations involving higher-order derivatives without diffusion is quite limited. 
Recent advancements have introduced a stable fully discrete LDG method known as the high-order RK-LDG method. Various studies have investigated the stability of this method in various contexts, as seen in \cite{sun2017stability}. Recently, Hunter et al. \cite{hunter2023stability} determined the stability of a fully discrete implicit-explicit RK method for the linearized KdV equation with periodic initial data using the Fourier method. However, to the best of our knowledge, the LDG method has not been developed for the Benjamin-Ono equation.

In this paper, our approach to design the LDG scheme for the Benjamin-Ono equation \eqref{BOeqn} involves the introduction of auxiliary variables to represent \eqref{BOeqn} into the system with lower-order derivatives.
A crucial part of this process involves constructing appropriate numerical fluxes at the interior interfaces, while boundary numerical fluxes are determined by the prescribed boundary conditions.
Additionally, to develop a fully discrete LDG scheme, we discretize time using the Crank-Nicolson method. We demonstrate that the fully discrete scheme is stable for any general nonlinear flux. To the best of our knowledge, this is the first study to develop and analyze a DG method for the Benjamin-Ono equation. The proposed LDG scheme thus introduces a new framework for stable and efficient numerical approximations of this nonlocal dispersive equation.
The main ingredients of the paper are enlisted below:
 \begin{enumerate}
     \item We design a LDG scheme for the Benjamin-Ono equation \eqref{BOeqn} and establish the stability and sub-optimal order of convergence of the devised semi-discrete scheme with a general nonlinear flux. The stability analysis has also been extended for fully discrete Crank-Nicolson scheme.
     \item Furthermore, we extend our methodology to include higher order temporal discretization of equation \eqref{BOeqn} using the classical four-stage fourth-order Runge-Kutta (RK) method \cite{sun2017stability}. For the linear case, we are able to show that the proposed fully discrete scheme is \emph{strongly stable} through the \emph{two-step} and \emph{three-step strong stability} estimates.
     \item  We demonstrate the rates obtained by numerical illustrations are optimal and it preserves the conserved quantity like mass and momentum in discrete set up.
 \end{enumerate}

The rest of the paper is organized as follows. We commence our investigation by introducing a few preliminary lemmas and a semi-discrete LDG scheme and its stability and error analysis in Section \ref{sec2}. We present the stability analysis of the fully discrete Crank-Nicolson LDG scheme for general nonlinear flux and the fully discrete fourth-order RK-LDG scheme for linear flux in Section \ref{sec4}.
The efficiency of the scheme is validated through some numerical examples presented in Section \ref{sec5}. Concluding remarks and a few remarks about future work are given in Section \ref{sec6}.

 
\section{Semi-discrete LDG scheme}\label{sec2}
\subsection{Preliminary results}
Hereby we describe a few relevant properties of the Hilbert transform through the following lemma. It is worthwhile to mention that these properties are instrumental for subsequent analysis. Let $\mathcal{S}(\R)$ denote the Schwartz space.
     \begin{lemma}{(See \cite[Chapter 15]{king2009hilbert})}\label{Pilbert}
     Let $\phi\in \mathcal{S}(\R)$. Then the Hilbert transform $\mathcal H$ satisfies the following properties:
         \begin{enumerate}[label=\roman*)]
             \item Skew symmetric:
              \begin{equation*}
                  (\mathcal{H}\phi_1 , \phi_2) = - (\phi_1,\mathcal{H}\phi_2),\qquad \forall \phi_1,\phi_2 \in L^2(\R).
              \end{equation*}
             \item Commutes with derivatives:
             \begin{equation*}
                 \mathcal{H}\phi_x = (\mathcal{H}\phi)_x.
             \end{equation*}
             \item $L^2$-isometry property:
             \begin{equation*}
                 \norm{\mathcal{H}\phi}_{L^2(\R)} = \norm{\phi}_{L^2(\R)}.
             \end{equation*}
             \item  Orthogonality:
             \begin{equation*}
                 (\mathcal{H}\phi,\phi) = 0.
             \end{equation*}
         \end{enumerate}
         where $(\cdot,\cdot)$ is the standard $L^2$-inner product.
     \end{lemma}


Given that the original problem is defined over the entire real line due to the involvement of a non-local operator $\mathcal{H}$. However, for the numerical purposes, following a similar approach in \cite{xu2014discontinuous,thomee1998numerical}, we restrict it to a sufficiently large bounded domain $\Omega:=[a,b]$, where $a<b$, such that $U$ has a compact support within $\Omega$ for all time $0<t<T$. Hence, it becomes imperative to impose the boundary conditions
$U(a,t) = 0 = U(b,t), \text{ for all  } t<T$.
Moreover, the properties of the Hilbert transform introduced in the Lemma \ref{Pilbert} remain applicable for a bounded domain $\Omega$, provided $\phi$ has a compact support within $\Omega$.

We partition the domain $\Omega$ into intervals $I_i = (x_{i-\frac{1}{2}},x_{i+\frac{1}{2}})$ with $a=x_{\frac{1}{2}}<x_{\frac{3}{2}}<\cdots<x_{N+\frac{1}{2}}=b$, where $N$ represents the number of elements. This partition creates a mesh of elements denoted by $\mathcal{I}$, with each element having a spatial step size $h_i = x_{i+\frac{1}{2}}-x_{i-\frac{1}{2}}$ and a maximum step size $h=\max\limits_{1\leq i\leq N}\{h_i\}$.
In conjunction with this mesh, we define the broken Sobolev spaces as follows:
\[ H^1(\Omega,\mathcal{I}) := \{v:\Omega\to \mathbb{R} \big\vert ~v|_{I_i}\in H^1(I_i),\,i=1, 2,\cdots,N\}; \]
and 
\[ L^2(\Omega,\mathcal{I}) := \{v:\Omega\to \mathbb{R} \big\vert ~ v|_{I_i}\in L^2(I_i),\,i=1, 2,\cdots,N\}. \]
Within this framework, we introduce the notation $v_{i+\frac{1}{2}}$ to represent the value of $v$ at the nodes $\{x_{i+\frac{1}{2}}\}$, and denote the one-sided limits as 
$$v^\pm_{i+\frac{1}{2}} = v(x^\pm_{i+\frac{1}{2}}):=\lim\limits_{x\to x^\pm_{i+\frac{1}{2}}} v(x).$$
We define the local inner product and local $L^2(I_i)$ norm as follows:
\begin{equation*}
    (u,v)_{I_i} = \int_{I_i} uv\,dx,\quad\quad \|u\|_{I_i} = (u,u)_{I_i}^{\frac{1}{2}}, \quad\quad (u,v) = \sum\limits_{i=1}^N (u,v)_{I_i} \quad \text{and }\quad \|u\|_{L^2(\Omega)} = \sum\limits_{i=1}^N\|u\|_{I_i}.
\end{equation*}
With all this preparation we introduce the auxiliary variables $P$ and $Q$ such that 
$$P= \mathcal{H}Q, \qquad Q = U_x.$$
As a consequence, the equation \eqref{BOeqn} can be represented in the following equivalent form of first order differential system
\begin{equation}
    \begin{split}\label{systemBO}
         U_t & = -(f(U)-P)_x,\\
    P &= \mathcal{H}Q,\\
    Q &= U_x.
    \end{split}
\end{equation}

Prior to introducing the LDG scheme, we assume that the exact solution $(U,P,Q)$ of the system \eqref{systemBO} belongs to
\[ \mathcal{T}_3\times K(\Omega,\mathcal{I}):= H^1(0,T;H^1(\Omega,\mathcal{I}))\times L^2(0,T;H^1(\Omega,\mathcal{I})) \times L^2(0,T;L^2(\Omega,\mathcal{I})). \]
This implies that the solution $(U,P,Q)$ of \eqref{systemBO} satisfies the following system:
\begin{equation}\label{exactsolun}
\begin{split}
     \left(U_t,v\right)_{I_i} & = \left(f(U) - P,v_x\right)_{I_i} - \left( f v -  P v\right)\big|_{x_{i-\frac{1}{2}}^+}^{x_{i+\frac{1}{2}}^-},\\
              \left(P,w\right)_{I_i} &= \left(\mathcal{H}Q,w\right)_{I_i},\\
              \left(Q,z\right)_{I_i}&=  -\left(U,z_x\right)_{I_i} +\left( U z\right)\big|_{x_{i-\frac{1}{2}}^+}^{x_{i+\frac{1}{2}}^-},
\end{split}
\end{equation}
for all $w\in L^2(\Omega,\mathcal{I})$, $v,z\in H^1(\Omega,\mathcal{I}) $, and for $i=1, 2,\cdots,N$.

We define the finite element $V^k\subset H^1(\Omega,\mathcal{I})$ by
\begin{equation}\label{elem_space}
    V^k = \{v\in L^2(\Omega):v|_{I_i}\in P^k (I_i),~ \forall i =1,2,\cdots,N\},
\end{equation}
where $P^k(I_i)$ is the space of polynomials of degree up to order $k$ $(\geq 1)$ on $I_i$. 
\subsection{LDG scheme}
To develop the LDG scheme for the Benjamin-Ono equation, it is necessary to define the numerical fluxes $\hat u,$ $\hat p$ and the nonlinear flux $\hat f$ at interfaces and boundaries. We introduce the following notations:
\begin{equation*}
   \{\!\!\{u\}\!\!\} = \frac{u^-+u^+}{2},\qquad \llbracket u \rrbracket = u^+-u^-.
\end{equation*}
We choose the alternative numerical flux which is given by 
\begin{align}\label{alterflux}
 \hat p = p^+,\qquad \hat u = u^-,
\end{align}
or alternatively,
\begin{align*}
 \hat p = p^-,\qquad \hat u = u^+,
\end{align*}
at interface $x_{i+\frac{1}{2}}$, $i=1,2, \cdots,N-1$. We set the boundary flux as 
\begin{equation}\label{Boundaryfluxeshilbert}
    \begin{split}
          \hat u_{N+\frac{1}{2}} = U(b,t)=0, &\qquad \hat u_{\frac{1}{2}}  = U(a,t)=0, ~\text{ for all } t<T,\\
    \hat p_{N+\frac{1}{2}} = p^-_{N+\frac{1}{2}}, &\qquad \hat p _{\frac{1}{2}} =p^+_{\frac{1}{2}}.\\
    \end{split}
\end{equation}
For the nonlinear flux $\hat f$, we can use any consistent and monotone flux \cite{yan2002local}. In particular, we consider the following  Lax-Friedrichs flux 
\begin{equation}\label{LF}
    \hat f = \hat f (u^-,u^+) = \frac{1}{2}\big(f(u^-) + f(u^+) -\delta \llbracket u\rrbracket\big),\qquad \delta = \max\limits_{u}|f'(u)|,
\end{equation}
where the maximum is taken over a range of $u$ in a relevant element.

Applying the DG approach in all the equations of the above system \eqref{exactsolun}, we design the scheme as follows: we seek an approximation $$(u,p,q)\in H^1(0,T;V^k)\times L^2(0,T;V^k)\times L^2(0,T;V^k) =: \mathcal{T}_3\times\mathcal{V}^k$$ to $(U,P,Q)$, where $U$ is an exact solution of \eqref{BOeqn} with $P = \mathcal{H}Q,~ Q = U_x ,$ such that for all test functions $(v,w,z)\in \mathcal{T}_3\times\mathcal{V}^k$ and $i=1,\cdots,N$, the following system of equations holds:
\begin{equation*}
\begin{split}
     \left(u_t,v\right)_{I_i} & = \left(f(u) - p,v_x\right)_{I_i} - \left(\hat f v - \hat p v\right)\Big|_{x_{i-\frac{1}{2}}^+}^{x_{i+\frac{1}{2}}^-},\\
              \left(p,w\right)_{I_i} &= \left(\mathcal{H}q,w\right)_{I_i},\\
              \left(q,z\right)_{I_i}&= -\left(u,z_x\right)_{I_i} + \left(\hat u z\right)\Big|_{x_{i-\frac{1}{2}}^+}^{x_{i+\frac{1}{2}}^-},\\
              \left(u^0,v\right)_{I_i} &= \left(U_0,v\right)_{I_i}.
\end{split}
\end{equation*}
The above system of equations can be rewritten as
\begin{equation}\label{LDGschemeBO}
\begin{split}
     \left(u_t,v\right)_{I_i} & =  \mathcal{F}_i(f(u),v) - \mathcal{D}_i^+(p,v),\\
              \left(p,w\right)_{I_i} &= \left(\mathcal{H}q,w\right)_{I_i},\\
              \left(q,z\right)_{I_i}&= -\mathcal{D}_i^-(u,z),\\
              \left(u^0,v\right)_{I_i} &= \left(U_0,v\right)_{I_i},
\end{split}
\end{equation}
where
\begin{align*}
    \mathcal{F}_i(f(u),v) &= \left(f(u),v_x\right)_{I_i} - \hat f{_{{i+\frac{1}{2}}}} v^-_{i+\frac{1}{2}} +  \hat f_{i-\frac{1}{2}} v^+_{i-\frac{1}{2}}, \qquad \text{ for } i=1,2,\cdots,N,\\
    \mathcal{D}_i^+(p,v) &= (p,v_x)_{I_i} - p^+_{{i+\frac{1}{2}}} v^-_{i+\frac{1}{2}} +   p^+_{i-\frac{1}{2}} v^+_{i-\frac{1}{2}}, \qquad \text{ for } i=1,2,\cdots,N-1,\\
    \mathcal{D}_i^-(u,z) &= (u,z_x)_{I_i} -  u^-_{{i+\frac{1}{2}}} v^-_{i+\frac{1}{2}} +  u^-_{i-\frac{1}{2}} v^+_{i-\frac{1}{2}}, \qquad \text{ for } i=2,3,\cdots,N-1,
\end{align*}
and
 \begin{align*}
    & \mathcal{D}_1^-(u,z) = (u,z_x)_{I_1} -  u^-_{{\frac{3}{2}}} v^-_{\frac{3}{2}},\quad \mathcal{D}_N^-(u,z) = (u,z_x)_{I_N} + u^-_{N-\frac{1}{2}} v^+_{N-\frac{1}{2}}, \\ &  
    \mathcal{D}_N^+(p,v) = (p,v_x)_{I_N} - p^-_{{N+\frac{1}{2}}} v^-_{N+\frac{1}{2}} +   p^+_{N-\frac{1}{2}} v^+_{N-\frac{1}{2}}.
 \end{align*}

The proposed LDG scheme \eqref{LDGschemeBO} for the Benjamin-Ono equation works in the following way: given $u$, we use the third equation of \eqref{LDGschemeBO} to obtain $q$ locally; more precisely, $q$ in the cell $I_i$ can be computed with the information of $u$ in the cells $I_{i-1}$ and $I_i$. Afterwards, with the help of $q$ in the cell $I_i$, one can obtain $p$ locally in the cell $I_i$. Finally, we update the approximate solution $u$ in the cell $I_i$ incorporating $p$ and $u$ in the cells $I_{i-1}$ and $I_i$. In a similar way, for the choice of alternative fluxes $\hat p = p^-,~\hat u = u^+$ the algorithm can be adopted accordingly.

Summing equation \eqref{LDGschemeBO} over $i=1,2,\cdots,N$, we have
\begin{equation}\label{LDGschemeBO1}
\begin{split}
     \left(u_t,v\right) & =  \mathcal{F}(f(u),v) + \mathcal{D}^+(p,v),\\
              \left(p,w\right) &= \left(\mathcal{H}q,w\right),\\
              \left(q,z\right)&= \mathcal{D}^-(u,z),
\end{split}
\end{equation}
where
\[ \mathcal{F} = \sum\limits_{i=1}^N  \mathcal{F}_i\quad \text{and}\quad   \mathcal{D}^\pm = -\sum\limits_{i=1}^N  \mathcal{D}^\pm_i.   \]
Consequently, we represent $\mathcal{F}$, $\mathcal{D}^\pm$ and associated numerical fluxes in the following way:
\begin{align}
    \mathcal{F}(f(u),v) & = \sum\limits_{i=1}^N(f(u),v_x)_{I_i} - \sum\limits_{i=1}^N (\hat fv)|_{x_{i-\frac{1}{2}}^+}^{x_{i+\frac{1}{2}}^-} =(f(u),v_x)+ \hat f_{\frac{1}{2}}v^+_{\frac{1}{2}} - \hat f_{N+\frac{1}{2}}v^-_{N+\frac{1}{2}} + \sum\limits_{i=1}^{N-1} \hat f_{i+\frac{1}{2}}\llbracket v \rrbracket_{i+\frac{1}{2}},\label{Fscript}\\ 
    \mathcal{D}^+(p,v)  &= -(p,v_x)-p_{\frac{1}{2}}^+v^+_{\frac{1}{2}} +  p_{N+\frac{1}{2}}^-v^-_{N+\frac{1}{2}} -\sum\limits_{i=1}^{N-1} p^+_{i+\frac{1}{2}}\llbracket v\rrbracket_{i+\frac{1}{2}}, \label{Dplus}
\end{align}
and similarly,
\begin{align}\label{upsum}
   \mathcal{D}^-(u,z)=  -(u,z_x)- \sum\limits_{i=1}^{N-1} u^-_{i+\frac{1}{2}}\llbracket z\rrbracket_{i+\frac{1}{2}}.
\end{align}
Moreover, we have the following result. 
\begin{proposition}\label{D+D_prop}
Let $u$ and $p\in V^k$. Then, we have
    \begin{equation}\label{D+D-}
    \mathcal{D}^+(p,u) + \mathcal{D}^-(u,p)  = 0.
\end{equation}
\end{proposition}
\begin{proof}
We observe that integration by parts yields the following
    \begin{align*}
       (p,u_x) + (u,p_x) &= \sum\limits_{i=1}^N\left((p,u_x)_{I_i} + (u,p_x)_{I_i}\right) =  \sum\limits_{i=1}^N(up)|_{x_{i-\frac{1}{2}}^+}^{x_{i+\frac{1}{2}}^-} \\
        &=-u^+_{\frac{1}{2}}p^+_{\frac{1}{2}} + u^-_{N+\frac{1}{2}} p^-_{N+\frac{1}{2}} - \sum\limits_{i=1}^{N-1}u^-_{i+\frac{1}{2}} \llbracket p \rrbracket_{i+\frac{1}{2}} -\sum\limits_{i=1}^{N-1}p^+_{i+\frac{1}{2}} \llbracket u \rrbracket_{i+\frac{1}{2}}.
    \end{align*}
     From \eqref{Dplus} and \eqref{upsum}, we have
    \begin{align*}
        \mathcal{D}^+(p,u) + \mathcal{D}^-(u,p) & = -(p,u_x) - (u,p_x) - p_{\frac{1}{2}}^+u^+_{\frac{1}{2}} +  p_{N+\frac{1}{2}}^-u^-_{N+\frac{1}{2}} -\sum\limits_{i=1}^{N-1} p^+_{i+\frac{1}{2}}\llbracket u\rrbracket_{i+\frac{1}{2}}\\&\qquad -  \sum\limits_{i=1}^{N-1} u^-_{i+\frac{1}{2}}\llbracket p\rrbracket_{i+\frac{1}{2}}
         =0.
    \end{align*}
Hence the result follows.
\end{proof}

To carry out the further analysis of the LDG scheme \eqref{LDGschemeBO}, we define the corresponding compact form
    \begin{align}\label{perturbationBO}
        \nonumber \mathcal{B}(u,p,q;v,w,z) &= \sum\limits_{i=1}^N\Big[\left(u_t,v\right)_{I_i} +\left(p,w\right)_{I_i} -  \left(\mathcal{H}q,w\right)_{I_i}
      + \left(q,z\right)_{I_i} \Big]\\ & \quad -\mathcal{F}(f(u),v) - \mathcal{D}^+(p,v) -\mathcal{D}^-(u,z) ,
    \end{align}
    for all $(u,p,q)\in \mathcal{T}_3\times K(\Omega,\mathcal{I})$ and $(v,w,z)\in \mathcal{T}_3\times\mathcal{V}^k$.


Hereby we analyze the stability of the proposed semi-discrete LDG scheme \eqref{LDGschemeBO} for \eqref{BOeqn}.
\subsection{Stability of semi-discrete scheme}
We prove the following stability lemma for general flux function using an appropriate compact form:
\begin{lemma}{($L^2$-stability)}\label{stablemmabo}
    Let $u,p,q$ be obtained from the LDG scheme \eqref{LDGschemeBO}. Then the LDG scheme \eqref{LDGschemeBO} for the Benjamin-Ono equation \eqref{BOeqn} is $L^2$-stable. We have the following estimate
\begin{equation}\label{stabboundBO}
    \norm{u(\cdot,T)}_{L^2(\Omega)} \leq C\norm{u^0}_{L^2(\Omega)},
\end{equation}
    for any $T>0$ and a constant $C$.
\end{lemma}

\begin{proof}
Given the compact form \eqref{perturbationBO} for the scheme \eqref{LDGschemeBO}, we choose the test functions $(v, w, z) = (u, -q , p)$ in \eqref{perturbationBO}, which yields
\begin{align}\label{perturbation3bo}
            \nonumber\mathcal{B}(u,p,q;u,-q,p) &= \sum\limits_{i=1}^N\Big[\left(u_t,u\right)_{I_i} -\left(p,q\right)_{I_i} +  \left(\mathcal{H}q,q\right)_{I_i}
      + \left(q,p\right)_{I_i} \Big]\\ & \quad -\mathcal{F}(f(u),u) - \mathcal{D}^+(p,u) -\mathcal{D}^-(u,p).
    \end{align}
Using the Proposition \ref{D+D_prop}, estimate \eqref{Fscript} and result from Lemma \ref{Pilbert} in \eqref{perturbation3bo}, we have
\begin{align}\label{perturbation4bo}
            \mathcal{B}(u,p,q;u,-q,p) &= \left(u_t,u\right) - \sum\limits_{i=1}^{N} \left(f(u), u_x\right)_{I_i} -\hat f_{\frac{1}{2}}u^+_{\frac{1}{2}} +\hat f_{N+\frac{1}{2}}u^-_{N+\frac{1}{2}} -\sum\limits_{i=1}^{N-1} \hat f_{i+\frac{1}{2}}\llbracket u \rrbracket_{i+\frac{1}{2}}.
    \end{align}
Let us define $F(u) = \int^u f(u)\,du$. Then we have
    \begin{equation}\label{F_ubo}
        \sum\limits_{i=1}^{N} \left(f(u), u_x\right)_{I_i} = \sum\limits_{i=1}^{N} F(u)|_{u^+_{i-\frac{1}{2}}}^{u^-_{i+\frac{1}{2}}} = - \sum\limits_{i=1}^{N-1} \llbracket F(u)\rrbracket_{i+\frac{1}{2}} - F(u)_{\frac{1}{2}} + F(u)_{N+\frac{1}{2}}.
    \end{equation}
Incorporating \eqref{F_ubo} in equation \eqref{perturbation4bo}, we have
    \begin{align}\label{perturbation5bo}
         \nonumber\mathcal{B}(u,p,q;u,-q,p)=\left(u_t,u\right)  &+\sum\limits_{i=1}^{N-1} \llbracket F(u)\rrbracket_{i+\frac{1}{2}} + F(u)_{\frac{1}{2}} - F(u)_{N+\frac{1}{2}}\\&  -\hat f_{\frac{1}{2}}u^+_{\frac{1}{2}}  +\hat f_{N+\frac{1}{2}}u^-_{N+\frac{1}{2}}- \sum\limits_{i=1}^{N-1} \hat f_{i+\frac{1}{2}}\llbracket u \rrbracket_{i+\frac{1}{2}}.
    \end{align}
Since $(u, p, q)$ satisfies the LDG scheme \eqref{LDGschemeBO}, then we have $$\mathcal{B}(u, p, q; v, w, z)=0,$$ for any $(v, w, z) \in V^k$. As a result, we get 
    \begin{align}\label{utu}
       \nonumber \left(u_t,u\right)_{L^2(\Omega)}  &+\sum\limits_{i=1}^{N-1} \llbracket F(u)\rrbracket_{i+\frac{1}{2}} + F(u)_{\frac{1}{2}} - F(u)_{N+\frac{1}{2}}\\&  -\hat f_{\frac{1}{2}}u^+_{\frac{1}{2}}  +\hat f_{N+\frac{1}{2}}u^-_{N+\frac{1}{2}}- \sum\limits_{i=1}^{N-1} \hat f_{i+\frac{1}{2}}\llbracket u \rrbracket_{i+\frac{1}{2}}=0.
    \end{align}
    Since the numerical flux $\hat f = \hat f(u^-,u^+)$ is monotone, it is non-decreasing in its first argument and non-increasing in its second argument. As a consequence, we have $$\llbracket F(u)\rrbracket_{i+\frac{1}{2}} - \hat f_{i+\frac{1}{2}} \llbracket u\rrbracket_{i+\frac{1}{2}}>0,$$ for all $i=1,2,\cdots,N-1$. Dropping the positive term from the left-hand side of equation \eqref{perturbation5bo} and using the boundary conditions, we end up with
    \begin{equation}\label{perturbation6bo}
        \begin{split}
         \frac{1}{2}\frac{d}{dt}\norm{u}^2_{L^2(\Omega)}\leq  0.
        \end{split}
    \end{equation}
 Applying the Gronwall’s inequality, we obtain
    \begin{equation*}
    \norm{u(\cdot,T)}_{L^2(\Omega)} \leq C\norm{u^0}_{L^2(\Omega)}.
    \end{equation*}
    Hence the result follows.
    \end{proof}

\subsection{Error analysis of the scheme}
To advance with the error estimates, we define the special projection operators into the finite element space \( V^k \) as follows. For any sufficiently smooth \( g \), we define:
\begin{equation}\label{projectionprop}
    \begin{split}
        \int_{I_i}\big(\mathcal{P}^- g(x) - g(x)\big)y(x)\,dx &= 0 \quad \forall ~y \in P^{k-1}(I_i), \quad  ~\text{ and }~ (\mathcal{P}^- g)^-_{i+1/2} = g(x^-_{i+1/2}),\\
        \int_{I_i}\big(\mathcal{P} g(x) - g(x)\big)y(x)\,dx &= 0 \quad \forall~ y \in P^k(I_i),
    \end{split}
\end{equation}
for all \( i=1,2,\ldots,N \). Here \( \mathcal{P}^- \) represents the projection defined above and \( \mathcal{P} \) is the standard \( L^2 \) projection. Let \( U \) be the exact solution to \eqref{BOeqn}, and let \( u \) be the approximate solution obtained through the LDG scheme \eqref{LDGschemeBO}. We introduce some compact notations for the terms involving the difference between projections and approximations by the following
$$\mathcal{P}_h^-u = \mathcal{P}^-U - u, \quad \mathcal{P}_hq = \mathcal{P}Q - q, \quad \mathcal{P}_hp = \mathcal{P}P - p,$$
and 
$$\mathcal{P}^-_eU = \mathcal{P}^-U - U, \quad \mathcal{P}_eQ = \mathcal{P}Q - Q, \quad \mathcal{P}_eP = \mathcal{P}P - P.$$ 
Note that subscripts $h$ and $e$ in the above notations indicate the differences with approximations and exact solutions respectively.

In the process of establishing the error estimate for the equation \eqref{BOeqn}, we introduce several lemmas concerning the relationship between the physical flux \( f \) and the numerical flux \( \hat{f} \).

\begin{lemma}[See Lemma 3.1 in \cite{zhang2004error}]\label{Shulemma1}
    Let $\xi\in L^2(\Omega)$ be any piecewise smooth function and $f$ be a nonzero $C^1$ flux. On each interface of elements and on the boundary points we define
    \begin{align*}
        \beta(\hat f;\xi) = \begin{cases}
           (f(\{\!\!\{\xi\}\!\!\}) - \hat f(\xi)) \llbracket\xi\rrbracket^{-1},\qquad &if~\llbracket \xi \rrbracket \neq 0,\\
            \frac{1}{2}|f'(\{\!\!\{\xi\}\!\!\})|, \qquad &if~\llbracket \xi \rrbracket = 0,
        \end{cases}
    \end{align*}
    where $\hat f(\xi)$ is a consistent and monotone numerical flux. Then $\beta(\hat f;\xi)$ is bounded and positive. 
\end{lemma}
Borrowing the idea from \cite{xu2007error}, we look to estimate the nonlinear part $f(u)$ by defining 
\begin{align}
    \nonumber\sum\limits_{i=1}^N \mathcal{G}_i(f;U,u,v) &= \sum\limits_{i=1}^N \int_{I_i} \big(f(U) - f(u)\big)v_x\,dx + \sum\limits_{i=1}^N \Big(\big(f(U) - f(\{\!\!\{u\}\!\!\})\big)\llbracket v\rrbracket\Big)_{i+\frac{1}{2}}\\
   \label{Nonlinearpart} &\quad  + \sum\limits_{i=1}^N \Big((f(\{\!\!\{u\}\!\!\}) - \hat f)\llbracket v\rrbracket\Big)_{i+\frac{1}{2}}.
\end{align}
\begin{lemma}[See Corollary 3.6 in \cite{xu2007error}]\label{Shulemma2}
    Let $f\in C^3(\Omega)$ and let the operator $\mathcal{G}_i$ be defined by \eqref{Nonlinearpart}. Then we have the following estimate:
    \begin{align}\label{est_f}
        \nonumber\sum\limits_{i=1}^N \mathcal{G}_i(f;U,u,v) &\leq -\frac{1}{4}\beta(\hat f;u)\sum\limits_{i=1}^N\llbracket v\rrbracket^2_{i+\frac{1}{2}} +\big(C+C_{\ast}h^{-1}\norm{U-u}_{L^\infty(\Omega)}^2\big)h^{2k+1} \\
        &\qquad +\big(C+C_{\ast}(\norm{v}_{L^\infty(\Omega)} + h^{-1}\norm{U-u}_{L^\infty(\Omega)}^2)\big)\norm{v}_{L^2(\Omega)}^2.
    \end{align}
\end{lemma}

We deal with the nonlinear flux $f(u)$ by making an $a~priori$ assumption \cite{xu2007error}. Let $h$ be small enough and for $k\geq 1$, there holds
\begin{equation}\label{prioriass}
    \norm{U-u}_{L^2(\Omega)} \leq h,
\end{equation}
where $u\in V^k$ is an approximation of $U$.
The above assumption is unnecessary for linear flux $f(u) = u$.
We define the bilinear operator $\mathcal{B}_0$ by the following
\begin{align}\label{Bilinear_BO}
           \nonumber \mathcal{B}_0(u,p,q;v,w,z) &= \sum\limits_{i=1}^N\Big[\left(u_t,v\right)_{I_i} +\left(p,w\right)_{I_i} -  \left(\mathcal{H}q,w\right)_{I_i}
      + \left(q,z\right)_{I_i} \Big]  - \mathcal{D}^+(p,v) -\mathcal{D}^-(u,z)\\
          \nonumber &= \sum\limits_{i=1}^N\Big[\left(u_t,v\right)_{I_i} + \left(p,v_x\right)_{I_i}  +\left(p,w\right)_{I_i} -  \left(\mathcal{H}q,w\right)_{I_i}
          + \left(q,z\right)_{I_i} + \left(u,z_x\right)_{I_i}\Big] \\&\quad+ \mathcal{IF}_0(u,p;v,z),
    \end{align}
     where the term $\mathcal{IF}_0$ is given by
 \begin{equation}\label{Bilearflux_BO}
 \begin{split}
        \mathcal{IF}_0(u,p;v,z) := p^+_{\frac{1}{2}}v^+_{\frac{1}{2}}  - p^-_{N+\frac{1}{2}}v^-_{N+\frac{1}{2}} + \sum\limits_{i=1}^{N-1} p^+_{i+\frac{1}{2}}\llbracket v \rrbracket_{i+\frac{1}{2}}
       + \sum\limits_{i=1}^{N-1}u^-_{i+\frac{1}{2}} \llbracket z \rrbracket_{i+\frac{1}{2}}.
  \end{split}
    \end{equation}
    Note that $\mathcal{B}_0 = \mathcal{B}$ if we take $f=0$ in the definition of $\mathcal{B}$ in \eqref{perturbationBO}, that is, $\mathcal{B}_0$ is the linear part of $\mathcal{B}$. Incorporating \eqref{perturbation5bo} in \eqref{Bilinear_BO}, we have
\begin{equation}\label{Bilearflux_BO1}
        \begin{split}
            \mathcal{B}_0(u,p,q;u,-q,p)& = \left(u_t,u\right)_{L^2(\Omega)}.
        \end{split}
    \end{equation}

\begin{theorem}\label{NLerrorBO}
 Let \( U \in C^{k+1}(\Omega) \) be an exact solution of \eqref{BOeqn} and \( u \) be an approximate solution obtained using the LDG scheme \eqref{LDGschemeBO}. Furthermore, assume that $f\in C^3(\Omega)$. Then, for sufficiently small \( h \), the following error estimate holds
\begin{equation}\label{errestimateNLBO}
    \norm{U-u}_{L^2(\Omega)} \leq Ch^{k+1/2},
\end{equation}
where \( C \) is a constant depending on the fixed time \( T > 0 \), \( k\geq1 \), and the bounds on the derivatives \( |f^{(m)}| \) for \( m = 1, 2, 3 \).
\end{theorem}

\begin{proof}
      We begin by deriving an error equation. Since \(U\) is an exact solution of \eqref{BOeqn}, we define
    \begin{align*}
    P = Q, \qquad Q = U_x.
    \end{align*}
     Then \(U\), \(P\), and \(Q\) satisfy the equation \eqref{perturbationBO}. Hence for any \((v,w,z)\in V^k\), we have
    \begin{align*}
        \mathcal{B}(U,P,Q;v,w,z) = \mathcal{B}(u,p,q;v,w,z) =0,
    \end{align*}
    where \((u,p,q)\) is an approximate solution obtained by the scheme \eqref{perturbationBO}. By incorporating the bilinear operator \(\mathcal{B}_0\), we get
    \begin{align*}
         0&=\mathcal{B}(U,P,Q;v,w,z) - \mathcal{B}(u,p,q;v,w,z)\\
          &=\mathcal{B}_0(U,P,Q;v,w,z) - \mathcal{B}_0(u,p,q;v,w,z)-\sum\limits_{i=1}^N \mathcal{G}_i(f;U,u,v)\\
          &=\mathcal{B}_0(U-u,P-p,Q-q;v,w,z) - \sum\limits_{i=1}^N \mathcal{G}_i(f;U,u,v).
    \end{align*}
    Taking into account the projection operators \(\mathcal{P}^-\) and \(\mathcal{P}\) defined in \eqref{projectionprop}, we choose \((v,w,z) = (\mathcal{P}_h^-u, -\mathcal{P}_hq, \mathcal{P}_hp)\). Since 
    $$U-u =\mathcal{P}^-_hu -  \mathcal{P}^-_eU, \quad P-p =\mathcal{P}_hp -  \mathcal{P}_eP \quad\text{ and } \quad Q-q =\mathcal{P}_hq -  \mathcal{P}_eQ,$$ we have
    \begin{equation}\label{A_err}
        \begin{split}
             \mathcal{B}_0(\mathcal{P}_h^-u,\mathcal{P}_hp,\mathcal{P}_hq;\mathcal{P}_h^-u, -\mathcal{P}_hq, \mathcal{P}_hp)
         &= \mathcal{B}_0(\mathcal{P}^-_eU,\mathcal{P}_eP,\mathcal{P}_eQ;\mathcal{P}_h^-u, -\mathcal{P}_hq, \mathcal{P}_hp)\\&\quad+ \sum\limits_{i=1}^N \mathcal{G}_i(f;U,u,\mathcal{P}_h^-u).
        \end{split}
    \end{equation}
     We estimate the first term on the right-hand side of \eqref{A_err}. From equation \eqref{Bilinear_BO}, we have
       \begin{equation}\label{Bilinear_BOerr}
        \begin{split}
            \mathcal{B}_0(\mathcal{P}^-_eU,\mathcal{P}_eP,\mathcal{P}_eQ;\mathcal{P}_h^-u, -\mathcal{P}_hq, \mathcal{P}_hp) &= \sum\limits_{i=1}^N\Big[\left((\mathcal{P}^-_eU)_t,\mathcal{P}_h^-u\right)_{I_i}  + \left(\mathcal{P}_eP,(\mathcal{P}_h^-u)_x\right)_{I_i} \\&\quad -\left(\mathcal{P}_eP,\mathcal{P}_hq\right)_{I_i} +  \left(\mathcal{H}\mathcal{P}_eQ,\mathcal{P}_hq\right)_{I_i}
         \\&\quad + \left(\mathcal{P}_eQ,\mathcal{P}_hp\right)_{I_i} + \left(\mathcal{P}^-_eU,(\mathcal{P}_hp)_x\right)_{I_i}\Big] \\&\quad+ \mathcal{IF}_0(\mathcal{P}^-_eU,\mathcal{P}_eP;\mathcal{P}_h^-u,\mathcal{P}_hp).
        \end{split}
    \end{equation}
    Since we have
 \begin{align*}
     (\mathcal{P}_hp)_x \in P^{k-1}(I_i),~(\mathcal{P}_h^-u)_x \in P^{k-1}(I_i),\quad \text{and} \quad   \mathcal{P}_hp,~\mathcal{P}_hq\in P^{k}(I_i),
 \end{align*}
and consequently, from the projection properties defined in \eqref{projectionprop} implies
\begin{align*}
   \left(\mathcal{P}_eP,(\mathcal{P}_h^-u)_x\right)_{I_i} =0 , \quad \left(\mathcal{P}_eQ,\mathcal{P}_hp\right)_{I_i} =0 , \quad  \left(\mathcal{P}^-_eU,(\mathcal{P}_hp)_x\right)_{I_i} =0 ,
\end{align*}
for all \(i=1, 2,\cdots,N\), and $(\mathcal{P}^-_eU)^-_{i+\frac{1}{2}} = 0$ for all \(i=1, 2,\cdots,N-1\). Furthermore, we observe that 
\begin{equation}
    \sum\limits_{i=1}^N\left(\mathcal{P}_eP-\mathcal{H}\mathcal{P}_eQ,\mathcal{P}_hq\right)_{I_i}\leq \|(\mathcal{P}P-P)-\mathcal{H}(\mathcal{P}Q-Q)\|_{L^2(\Omega)}\|\mathcal{P}_hq\|_{L^2(\Omega)}\leq C(\varepsilon)h^{2k+2} + \varepsilon\|\mathcal{P}_hq\|_{L^2(\Omega)}^2,
\end{equation}
where we have used the isometry of the Hilbert transform from Lemma \ref{Pilbert} and the projection property, and  the Young's inequality. Moreover, from the approximation theory on the point values associated with the projection operators \cite[Section 3.2]{Ciarlet}, we have
\begin{align*}
    \left(\mathcal{P}_eP\right)^+_{i+\frac{1}{2}} \leq Ch^{k+1},\quad (\mathcal{P}^-_eU)^-_{i-\frac{1}{2}} \leq Ch^{k+1},
\end{align*}
for all $i=1, 2,\cdots,N-1$, and 
\begin{align*}
    (\mathcal{P}_h^-u)^-_{N+\frac{1}{2}} = 0, \quad (\mathcal{P}_h^-u)^+_{\frac{1}{2}}= 0,
\end{align*}
by using boundary conditions.
Combining these error bounds and using the Young's inequality, we have 
\begin{align*}
    \mathcal{IF}_0(\mathcal{P}^-_eU,\mathcal{P}_eP;\mathcal{P}_h^-u,\mathcal{P}_hp) =&  (\mathcal{P}_eP)^+_{\frac{1}{2}}(\mathcal{P}_h^-u)^+_{\frac{1}{2}}  - (\mathcal{P}_eP)^-_{N+\frac{1}{2}}(\mathcal{P}_h^-u)^-_{N+\frac{1}{2}}  \\&+\sum\limits_{i=1}^{N-1} (\mathcal{P}_eP)^+_{i+\frac{1}{2}}\llbracket \mathcal{P}_h^-u \rrbracket_{i+\frac{1}{2}}
       + \sum\limits_{i=1}^{N-1}(\mathcal{P}^-_eU)^-_{i+\frac{1}{2}} \llbracket \mathcal{P}_hp\rrbracket_{i+\frac{1}{2}}\\
     \leq&  \sum\limits_{i=1}^{N-1} \Big(C(\varepsilon)\big((\mathcal{P}_eP)^+_{i+\frac{1}{2}}\big)^2 + \varepsilon\llbracket \mathcal{P}_h^-u \rrbracket_{i+\frac{1}{2}}^2\Big)\\
     \leq& C(\Omega)h^{2k+1} + \varepsilon\sum\limits_{i=1}^{N-1} \llbracket \mathcal{P}_h^-u \rrbracket_{i+\frac{1}{2}}^2,
\end{align*}
where $\varepsilon >0$. Using the above estimates in \eqref{Bilinear_BOerr}, we end up with 
\begin{equation}\label{Bilinear_BOerr2}
        \begin{split}
            \mathcal{B}_0(\mathcal{P}^-_eU,\mathcal{P}_eP,\mathcal{P}_eQ;\mathcal{P}_h^-u, -\mathcal{P}_hq, \mathcal{P}_hp) &\leq\left((\mathcal{P}^-_eU)_t,\mathcal{P}_h^-u\right)_{L^2(\Omega)}+ C(\Omega)h^{2k+1} + \varepsilon\sum\limits_{i=1}^{N-1} \llbracket \mathcal{P}_h^-u \rrbracket_{i+\frac{1}{2}}^2.
        \end{split}
    \end{equation}
 Incorporating estimates from \eqref{Bilearflux_BO1} and  \eqref{Bilinear_BOerr2} in \eqref{A_err}, we have

\begin{align}\label{Bilinear_BOerr34}
           \nonumber \left((\mathcal{P}_h^-u)_t,\mathcal{P}_h^-u\right)_{L^2(\Omega)}
           \leq& \left((\mathcal{P}^-_eU)_t,\mathcal{P}_h^-u\right)_{L^2(\Omega)}+ C(\Omega) h^{2k+1}  + \varepsilon\sum\limits_{i=1}^{N-1} \llbracket \mathcal{P}_h^-u \rrbracket_{i+\frac{1}{2}}^2\\&\quad +\varepsilon\|\mathcal{P}_hq\|_{L^2(\Omega)}^2 + \sum\limits_{i=1}^N \mathcal{G}_i(f;U,u,\mathcal{P}_h^-u).
\end{align}
With the help of estimate \eqref{est_f} and as $\varepsilon$ is arbitrarily small, the estimate \eqref{Bilinear_BOerr34} reduces to
\begin{equation}\label{Bilinear_BOerr3}
        \begin{split}
            &\left((\mathcal{P}_h^-u)_t,\mathcal{P}_h^-u\right)_{L^2(\Omega)}+\frac{1}{4}\beta(\hat f;\mathcal{P}_h^-u)\sum\limits_{i=1}^N\llbracket \mathcal{P}_h^-u\rrbracket^2_{i+\frac{1}{2}} \\
            &\qquad\leq \left((\mathcal{P}^-_eU)_t,\mathcal{P}_h^-u\right)_{L^2(\Omega)}+ C(\Omega) h^{2k+1} + \Big(C+C_{\ast}h^{-1}\norm{U-u}_{L^\infty(\Omega)}^2\Big)h^{2k+1}\\
            &\qquad\qquad +\Big(C+C_{\ast}\big(\norm{\mathcal{P}_h^-u}_{L^\infty(\Omega)} + h^{-1}\norm{U-u}_{L^\infty(\Omega)}^2\big)\Big)\norm{\mathcal{P}_h^-u}_{L^2(\Omega)}^2 .
        \end{split}
    \end{equation}
Utilizing the inverse inequality $\norm{u}_{L^\infty(\Omega)}\leq h^{-1/2}\norm{u}_{L^2(\Omega)}$ and a priori assumption \eqref{prioriass}, we obtain the estimate
\begin{equation}\label{estnonl}
    h^{-1}\norm{U-u}_{L^{\infty}(\Omega)}^2h^{2k+1} \leq h^{-2} \norm{U-u}^2_{L^2(\Omega)}h^{2k+1}\leq h^{2k+1}.
\end{equation}
  Using  the above estimate and the positivity of $\beta$ from Lemma \ref{Shulemma1}, equation \eqref{Bilinear_BOerr3} implies
    \begin{align*}
         &\frac{1}{2}\frac{d}{dt}\norm{\mathcal{P}_h^-u}^2_{L^2(\Omega)}\leq \left((\mathcal{P}^-_eU)_t,\mathcal{P}_h^-u\right)_{L^2(\Omega)} + C({\Omega})h^{2k+1} +C \norm{\mathcal{P}_h^-u}^2_{L^2(\Omega)}.
    \end{align*}
Using the standard approximation theory associated to the projection \cite[Section 3.2]{Ciarlet}, we have $\norm{\mathcal{P}_h^-u(\cdot,0)}^2_{L^2(\Omega)}=0$. Finally, with the help of Gronwall’s inequality, we obtain the estimate 
     \begin{equation*}
    \norm{U-u}_{L^2(\Omega)}\leq Ch^{k+1/2}.
\end{equation*}
This completes the proof.
    \end{proof}

\section{Stability of fully discrete LDG scheme}\label{sec4}
\subsection{Stability analysis of Crank-Nicolson fully discrete LDG scheme}
To develop the fully discrete local discontinuous Galerkin (LDG) scheme, we utilize the Crank-Nicolson method for time discretization in \eqref{LDGschemeBO1}. We partition the time domain using a time step \(\tau \). Let \(\{t_n = n\tau \}_{n=0}^{M}\) be the partition of the given time interval \([0, T]\), and define \(u^n = u(t_n)\) and \(u^{n+\frac{1}{2}} = \frac{1}{2}(u^{n+1} + u^n)\).

The fully discrete LDG scheme is designed as follows: Given \(u^n\) find \(u^{n+1}\) such that the following system holds
\begin{equation}\label{fdLDGschemeBO1}
\begin{split}
    \left(u^{n+1}, v\right) &= \left(u^{n}, v\right) + 
    \tau  \mathcal{F}(f(u^{n+\frac{1}{2}}), v) + \tau  \mathcal{D}^+(p^{n}, v),\\
    \left(p^{n}, w\right) &= \left(\mathcal{H} q^{n}, w\right),\\
    \left(q^{n}, z\right) &= \mathcal{D}^-(u^{n+\frac{1}{2}}, z),
\end{split}
\end{equation}
for all \(v, w, z \in V^k\) and $n=1,2,\cdots,M-1$, and set $u^0 = \mathcal{P}^-u_0$. Numerical fluxes at interfaces and boundaries involved in the above scheme can be defined in a similar way as in \eqref{alterflux}-\eqref{LF}.
\begin{lemma}
    The fully discrete scheme \eqref{fdLDGschemeBO1} is \(L^2\)-stable, and the solution \(u^n\) obtained by the scheme \eqref{fdLDGschemeBO1} satisfies
    \begin{equation}\label{stabl2}
        \|u^{n+1}\|_{L^2(\Omega)} \leq \|u^{n}\|_{L^2(\Omega)}, \qquad \forall n.
    \end{equation}
\end{lemma}
\begin{proof}
 We choose test functions \((v, w, z) = (u^{n+\frac{1}{2}}, -q^{n}, p^{n})\) in \eqref{fdLDGschemeBO1}, and adding all three equations yields
\begin{align*}
   \left(\frac{u^{n+1} - u^n}{\tau }, u^{n+\frac{1}{2}}\right)&= 
\mathcal{F}(f(u^{n+\frac{1}{2}}), u^{n+\frac{1}{2}}) +  \mathcal{D}^+(p^{n}, u^{n+\frac{1}{2}})+ \left(p^{n}, q^{n}\right) - \left(\mathcal{H} q^{n}, q^{n}\right)\\&\qquad-
    \left(q^{n}, p^{n}\right) + \mathcal{D}^-(u^{n+\frac{1}{2}}, p^{n}).
\end{align*}
 Using Lemma \ref{Pilbert}, Proposition \ref{D+D_prop} and monotonicity of the numerical flux \(\hat{f}^{n+\frac{1}{2}}\) in the above equation, we obtain
    \begin{align*}
        \left(\frac{u^{n+1} - u^n}{\tau }, u^{n+\frac{1}{2}}\right) \leq 0.
    \end{align*}
    This implies
    \begin{equation*}
        \|u^{n+1}\|_{L^2(\Omega)} \leq \|u^{n}\|_{L^2(\Omega)}, \qquad \forall n=0,1,\dots,M-1.
    \end{equation*}
    This completes the proof.
\end{proof}

\subsection{Stability analysis of higher order fully discrete LDG scheme}

Hereby we focus on the stability analysis with an explicit fourth-order Runge-Kutta time discretization. The spatial stability of the LDG scheme is demonstrated in Section \ref{sec2}. It is worth mentioning that the stability analysis under the higher-order time discretizations is a challenging task, and at present we have pursued this analysis exclusively in the linear case. In particular, for simplicity, we choose the zero flux function for the subsequent analysis \cite{wang2024analysis}. We would like to remark that to the best of our knowledge, the stability analysis for fourth-order LDG schemes involving general LDG operators has not been performed in the literature. The following analysis can be considered as a contribution to the stability analysis under the aforementioned assumptions.

Then the semi-discrete scheme \eqref{LDGschemeBO} corresponds to an ODE system
         \begin{equation}\label{ddtldg}
          \frac{d}{dt}u = L_h u.
         \end{equation}
where LDG operator $L_h$ can be recovered from the semi-discrete scheme \eqref{LDGschemeBO}.
We consider the four-stage explicit fourth-order RK method for time discretization \cite{sun2017stability}:
\begin{align}\label{RK4scheme}
    u^{n+1} = P_4(\tau L_h)u^n,
\end{align}
where the operator $P_4$ is given by
\begin{align*}
    P_4(\tau L_h) = I+\tau L_h +\frac{1}{2}(\tau L_h)^2 +\frac{1}{6}(\tau  L_h)^3 + \frac{1}{24}( \tau L_h)^4.
\end{align*}
We introduce a few notations which will be used frequently in subsequent stability analysis. 
We say that the LDG operator $L_h$ is semi-negative if $(v,(L_h+L_h^T)v)\leq 0$, $\forall v \in V^k$ and it is denoted by $L_h+L_h^T\leq 0$. 
We define $[u,v] := - (u,(L_h+L_h^T)v)$. It is a bilinear form on $V^k$.
For convenience, we denote $\mathcal{L}_h:=\tau  L_h$.
We introduce the notion of strong stability for the fully discrete scheme.
    \begin{definition}[Strong stability]
            Let $u^n$ be the approximate solution obtained by the fully discrete scheme \eqref{RK4scheme}. Then the scheme \eqref{RK4scheme} is said to be strongly stable if there exists an integer $n_0$, such that 
        \begin{equation}\label{strongstabdefn}
                \norm{u^{n}}_{L^2(\Omega)}\leq \norm{u^0}_{L^2(\Omega)}, \qquad \forall n\geq n_0.
        \end{equation}
\end{definition}
Hereby we state the following result associated to energy estimate which will be instrumental for further stability analysis.
\begin{lemma}[Energy equality \cite{sun2017stability}]\label{energyen}
Let $u^n$ be the solution of fully discrete scheme \eqref{RK4scheme}. Then
    \begin{equation*}
        \norm{u^{n+1}}^2_{L^2(\Omega)}-\norm{u^{n}}^2_{L^2(\Omega)} = \mathcal{Q}(u^n), \qquad \forall n\geq 1,
    \end{equation*}
    where
    \begin{equation*}
        \mathcal{Q}(u^n) = \frac{1}{576}\norm{\mathcal{L}_h^4u^{n}}^2_{L^2(\Omega)}-\frac{1}{72}\norm{\mathcal{L}_h^3u^{n}}^2_{L^2(\Omega)}+\tau \sum\limits_{i,j=0}^{3}\alpha_{ij}[\mathcal{L}_h^iu^n,\mathcal{L}_h^ju^n],
    \end{equation*}
    and
    \begin{equation*}
        A =\left(\alpha_{ij}\right)_{i,j=0}^{3}= -\begin{pmatrix}
1 & 1/2 & 1/6 & 1/24\\
1/2 & 1/3 & 1/8  &  1/24\\
1/6 & 1/8 & 1/24 & 1/48 \\
1/24 & 1/24 & 1/48 & 1/144
\end{pmatrix}.
    \end{equation*}
\end{lemma}
In Lemma \ref{energyen}, $\mathcal{Q}(u^n)$ is referred as the energy change of the approximate solution, consisting of two components: the numerical dissipation $\frac{1}{576}\norm{\mathcal{L}_h^4u^{n}}^2_{L^2(\Omega)}-\frac{1}{72}\norm{\mathcal{L}_h^3u^{n}}^2_{L^2(\Omega)}$ and the quadratic form $\tau \sum\limits_{i,j=0}^{3}\alpha_{ij}[\mathcal{L}_h^iu^n,\mathcal{L}_h^ju^n]$. The next lemma establishes the negativity of the quadratic form, also provides conditions for the strong stability.
\begin{lemma}[See Lemma 2.4 in \cite{sun2017stability}]\label{negativeQ}
    Let $L_h$ be a semi-negative operator and 
    \begin{equation}
        \mathcal{Q}_1(u) = \zeta\norm{\mathcal{L}_h^3(u)}_{L^2(\Omega)} + \tau \sum\limits_{i,j=0}^m \bar \alpha_{ij} [\mathcal{L}_h^iu,\mathcal{L}_h^ju],
    \end{equation}
    where $\bar\alpha_{ij}=\bar\alpha_{ji}$ and $m\geq 2$. If $\zeta<0$ and $\bar A = \left(\bar\alpha\right)_{i,j=0}^2$ is negative definite, then there exists a constant $c_0>0$ such that $\mathcal{Q}_1(u)\leq 0$ provided $\norm{\mathcal{L}_h}\leq c_0$, where $c_0$ is independent of $\tau$ and $h$.
\end{lemma}

For simplicity, we are using uniform time stepping. Since the semi-discrete scheme \eqref{LDGschemeBO} is spatially stable from the Lemma \ref{stablemmabo}, that is 
\begin{equation*}
    \left(\frac{d}{dt}u,u\right) = (L_h u,u)\leq 0,
\end{equation*}
the LDG operator $L_h$ is semi-negative by the following
\begin{align*}
    \left(\frac{d}{dt}u,u\right) = (L_h u,u) = \frac{1}{2}(L_h u,u) + \frac{1}{2}(u,L_h u)
    = \frac{1}{2}((L_h+L_h^T) u,u)\leq 0,
\end{align*}
implies $L_h+L_h^T\leq 0$. 

The question of whether the classical fourth-order Runge-Kutta method is strongly stable or not remained open until Sun and Shu partially addressed it in \cite{sun2017stability}, where they provided a counterexample showing that this method is not always strongly stable for semi-negative operators. However, it is worth noting that the semi-negative operator provided in the counterexample in \cite{sun2017stability} is not a DG operator. We introduce a new approach to demonstrate the strong stability of the fourth-order, four-stage RK-LDG scheme \eqref{RK4scheme}. By proving stability of the fully discrete scheme over two and three time steps, we establish the foundation from which the strong stability result follows. More precisely, we prove the following:
\begin{theorem}\label{thm:RK4_stab_result}
     Let $L_h+L_h^T\leq 0$. Then the four-stage fourth-order RK LDG scheme \eqref{RK4scheme} is strongly stable. That is, we have 
    \begin{equation*}
        \norm{u^{n}}_{L^2(\Omega)}\leq  \norm{u^{0}}_{L^2(\Omega)}, \qquad \forall n\geq 2,
    \end{equation*}
    provided $\tau  \norm{L_h}\leq c_0$, where $c_0$ is a constant.
\end{theorem}
Before proceeding with the proof of Theorem \ref{thm:RK4_stab_result}, we establish the three-step strong stability result.


\begin{proposition}[Three-step strong stability]
    Let $L_h+L_h^T\leq 0$. Then the four-stage fourth-order RK scheme \eqref{RK4scheme} is strongly stable in three steps. Therefore, we have 
    \begin{equation}\label{three_step_esti}
        \norm{u^{n+3}}_{L^2(\Omega)}\leq  \norm{u^{n}}_{L^2(\Omega)}, \qquad n\geq0,
    \end{equation}
    provided $\tau  \norm{L_h}\leq c_0$, where $c_0$ is a constant.
\end{proposition}
\begin{proof}
    From the Lemma \ref{energyen}, energy equality implies
    \begin{equation}\label{h3}
        \norm{u^{n+3}}_{L^2(\Omega)}^2- \norm{u^{n}}_{L^2(\Omega)}^2 = \mathcal{Q}(u^{n+2})+\mathcal{Q}(u^{n+1})+\mathcal{Q}(u^{n}), 
    \end{equation}
    where $\mathcal{Q}(u^{n})$ is defined in Lemma \ref{energyen}.
    We find the estimates for $\mathcal{Q}(u^{n+1})$ and $\mathcal{Q}(u^{n+2})$ in terms of $u^n$ in its quadratic part by the following calculation:
    \begin{align*}
        \mathcal{Q}(u^{n+1}) &= \frac{1}{576}\norm{\mathcal{L}_h^4u^{n+1}}^2_{L^2(\Omega)}-\frac{1}{72}\norm{\mathcal{L}_h^3u^{n+1}}^2_{L^2(\Omega)}+\tau \sum\limits_{i,j=0}^{3}\alpha_{ij}[\mathcal{L}_h^iu^{n+1},\mathcal{L}_h^ju^{n+1}]\\
        &=\frac{1}{576}\norm{\mathcal{L}_h^4u^{n+1}}^2_{L^2(\Omega)}-\frac{1}{72}\norm{\mathcal{L}_h^3u^{n+1}}^2_{L^2(\Omega)}+\tau \sum\limits_{i,j=0}^{3}\alpha_{ij}[\mathcal{L}_h^iP_4(\mathcal{L}_h)u^{n},\mathcal{L}_h^jP_4(\mathcal{L}_h)u^{n}]\\
        &=\frac{1}{576}\norm{\mathcal{L}_h^4u^{n+1}}^2_{L^2(\Omega)}-\frac{1}{72}\norm{\mathcal{L}_h^3u^{n+1}}^2_{L^2(\Omega)}+\tau \sum\limits_{i,j=0}^{7}\Tilde\alpha_{ij}[\mathcal{L}_h^iu^{n},\mathcal{L}_h^ju^{n}]
    \end{align*}
We define the matrices $A_0$ and $A_1$ as
   \begin{equation*}
       A_0 = (\alpha_{ij})_{i,j=0}^2 = -\begin{pmatrix}
1 & 1/2 & 1/6 \\
1/2 & 1/3 & 1/8 \\
1/6 & 1/8 & 1/24
\end{pmatrix}, \qquad  A_{1} = (\Tilde\alpha_{ij})_{i,j=0}^2  = -\begin{pmatrix}
1 & 3/2 & 7/6 \\
3/2 & 7/3 & 15/8 \\
7/6 & 15/8 & 37/24 
\end{pmatrix}.
   \end{equation*}
Our interest in the first $3\times 3$ coefficient matrix as it is important to note that this is sufficient to apply the result in Lemma \ref{negativeQ}. While the complete matrix can be obtained and is not difficult, it involves a lengthy derivation. For brevity, we choose not to provide the complete matrix here, and we refer to \cite{sun2017stability} for detailed information.

In a similar way, we obtain
   \begin{align*}
        \mathcal{Q}(u^{n+2}) &= \frac{1}{576}\norm{\mathcal{L}_h^4u^{n+2}}^2_{L^2(\Omega)}-\frac{1}{72}\norm{\mathcal{L}_h^3u^{n+2}}^2_{L^2(\Omega)}+\tau \sum\limits_{i,j=0}^{3}\alpha_{ij}[\mathcal{L}_h^iu^{n+2},\mathcal{L}_h^ju^{n+2}]\\
        &=\frac{1}{576}\norm{\mathcal{L}_h^4u^{n+2}}^2_{L^2(\Omega)}-\frac{1}{72}\norm{\mathcal{L}_h^3u^{n+2}}^2_{L^2(\Omega)}+\tau \sum\limits_{i,j=0}^{7}\Tilde\alpha_{ij}[\mathcal{L}_h^iu^{n+1},\mathcal{L}_h^ju^{n+1}]\\
        &=\frac{1}{576}\norm{\mathcal{L}_h^4u^{n+2}}^2_{L^2(\Omega)}-\frac{1}{72}\norm{\mathcal{L}_h^3u^{n+2}}^2_{L^2(\Omega)}+\tau \sum\limits_{i,j=0}^{7}\Tilde\alpha_{ij}[\mathcal{L}_h^iP_4(\mathcal{L}_h)u^{n},\mathcal{L}_h^jP_4(\mathcal{L}_h)u^{n}]\\
        &=\frac{1}{576}\norm{\mathcal{L}_h^4u^{n+2}}^2_{L^2(\Omega)}-\frac{1}{72}\norm{\mathcal{L}_h^3u^{n+2}}^2_{L^2(\Omega)}+\tau \sum\limits_{i,j=0}^{11}\Hat\alpha_{ij}[\mathcal{L}_h^iu^{n},\mathcal{L}_h^ju^{n}],
    \end{align*}
    where 
    \begin{equation*}
        A_{2} = (\Hat\alpha_{ij})_{i,j=0}^2  = -\begin{pmatrix}
1 & 5/2 & 19/6 \\
5/2 & 19/3 & 57/8 \\
19/6 & 57/8 & 253/24 
\end{pmatrix}.
    \end{equation*}
Assuming $\norm{\mathcal{L}_h}\leq 2$, we have the following estimates
    \begin{align*}
        \frac{1}{576}\norm{\mathcal{L}_h^4u^{n+2}}^2_{L^2(\Omega)}-\frac{1}{72}\norm{\mathcal{L}_h^3u^{n+2}}^2_{L^2(\Omega)}&\leq 0,\\
        \frac{1}{576}\norm{\mathcal{L}_h^4u^{n+1}}^2_{L^2(\Omega)}-\frac{1}{72}\norm{\mathcal{L}_h^3u^{n+1}}^2_{L^2(\Omega)}&\leq 0,\\
        \frac{1}{576}\norm{\mathcal{L}_h^4u^{n}}^2_{L^2(\Omega)}-\frac{1}{72}\norm{\mathcal{L}_h^3u^{n}}^2_{L^2(\Omega)}&\leq -\frac{1}{144}\norm{\mathcal{L}_h^3u^n}^2_{L^2(\Omega)}.
    \end{align*}
Hence \eqref{h3} becomes
    \begin{align*}
        \norm{u^{n+3}}_{L^2(\Omega)}^2- \norm{u^{n}}_{L^2(\Omega)}^2 &\leq -\frac{1}{144}\norm{\mathcal{L}_h^3u^n}^2_{L^2(\Omega)} + \tau \sum\limits_{i,j=0}^{11}\Bar\alpha_{ij}[\mathcal{L}_h^iu^{n},\mathcal{L}_h^ju^{n}]=:\mathcal{Q}_1(u^n).
    \end{align*}
Afterwards, we define the matrix $A$ as
    \begin{equation*}
        A := (\Bar\alpha_{ij})_{i,j=0}^2 = (\alpha_{ij})_{i,j=0}^2+(\Tilde\alpha_{ij})_{i,j=0}^2+(\Hat\alpha_{ij})_{i,j=0}^2  = -\begin{pmatrix}
3 & 9/2 & 9/2 \\
9/2 & 9 & 73/8 \\
9/2 & 73/8 & 97/8 
\end{pmatrix}.
\end{equation*}
Since the eigenvalues of $A$ are $-21.9444, -1.64399$ and $-0.536623$, it follows that $A$ is negative definite. Applying the Lemma \ref{negativeQ}, we obtain $\mathcal{Q}_1(u^n)\leq 0$ As a consequence, we have
 \begin{equation*}
        \norm{u^{n+3}}_{L^2(\Omega)}\leq  \norm{u^{n}}_{L^2(\Omega)}.
    \end{equation*}
Hence the result follows.
\end{proof}
\begin{proof}[Proof of Theorem \ref{thm:RK4_stab_result}]
 Since the LDG operator is semi-negative, applying the two-step stability result from 
\cite[Theorem 2.1]{sun2017stability}, we have
 \begin{equation}\label{two_step_esti}
        \norm{u^{n+2}}_{L^2(\Omega)}\leq  \norm{u^{n}}_{L^2(\Omega)},
    \end{equation}
    provided $\tau  \norm{L_h}\leq c_0$, where $c_0$ is a constant.
The estimate \eqref{two_step_esti} along with the estimate \eqref{three_step_esti} combined together provide the desired stability estimate of the fully discrete LDG scheme \eqref{RK4scheme}.
\end{proof}

\section{Numerical Experiments}\label{sec5}

In this section, our goal is to validate the proposed LDG scheme \eqref{LDGschemeBO} for the Benjamin-Ono equation \eqref{BOeqn}. We begin our numerical validation using the Crank-Nicolson (CN) LDG scheme \eqref{fdLDGschemeBO1} with \(\tau  = 0.5h\) and \(k=1\). 
We employ the periodic one soliton and two soliton solutions of the Benjamin-Ono equation. It is important to note that since the implicit term appears in the nonlinear part of the scheme \eqref{fdLDGschemeBO1}, we use the Newton iteration to solve it at each time step.

Rate of convergence $R_E$ for errors is defined for each intermediate step between element numbers $N_1$ and $N_2$ as
\begin{equation*}
    R_E = \frac{\ln(E(N_1))-\ln(E(N_2))}{\ln(N_2)-\ln(N_1)},
\end{equation*}
where $E$ is a function of number of elements $N$ and represents the $L^2$-error. The Benjamin-Ono equation \eqref{BOeqn} possesses an infinite number of conserved quantities \cite{galtung2018convergent}. Hereby we consider the first two specific quantities known as $mass$ and $momentum$. With normalization, these quantities in the discrete set up can be expressed as follows:
\begin{align*}
     \mathbf C_1^h &:= \frac{\int_{\Omega} u\,dx}{\int_{\Omega} U_0\,dx},\qquad
     \mathbf C_2^h := \frac{\|u\|_{L^2(\Omega)}}{\norm{U_0}_{L^2(\Omega)}}.
\end{align*}
Our aim is to preserve these quantities in the discrete setup.

\subsection*{Example 1} 
We compare the approximate solution $u$ obtained by \eqref{fdLDGschemeBO1} with the exact periodic solution of the Benjamin-Ono equation \eqref{BOeqn} with $f(U) = \frac{1}{2}U^2$. For instance, we consider the solution of \eqref{BOeqn} from Thom{\'e}e et al.  \cite{thomee1998numerical} (also see \cite{dwivedi2024stability,dwivedi2024fully})
\begin{equation}\label{BOsolution}
U(x,t) = \frac{2c\delta}{1-\sqrt{1-\delta^2}\cos(c\delta(x-ct))}, \qquad \delta = \frac{\pi}{cL},
\end{equation}
where we choose  $L=15$, $c=0.25$. We compare the above exact solution at the final time $T=20$. This solution characterizes periodic solitary waves, exhibiting periodicity and amplitude determined by the parameters $L$ and $c$ respectively.
Table \ref{tab:error_tableBO_t1} presents the \(L^2\)-errors of the proposed fully discrete CN-LDG scheme \eqref{fdLDGschemeBO1} using a polynomial of degree one. The results confirm that the numerically obtained rates are optimal and that the approximate solution \(u\) converges to the exact solution \(U\).

 \begin{table}[htbp]
        \centering
        \begin{tabular}{||c|c|c|c|c||}
            \hline
           $N$ & $E$ & $R_E$ & $\mathbf C_1^h$ & $\mathbf C_2^h$ \\
            \hline
            \hline
              160  & 1.65e-02&&  1.04 &0.97 \\  
            & & 2.14 &  &\\  
             320 & 3.77e-03&  &1.05 &0.97\\
            & &2.07& & \\
             640 &8.98e-04&   &1.02 &0.98\\  
            &  &2.04&  & \\  
            1280 &2.18e-04& &1.00 &1.00\\
            \hline
        \end{tabular}
        \caption{$L^2$-error and rate of convergence $R_E$ for the CN-LDG scheme at time $T=20$ taking $N$ elements and polynomial degree $k=1$.}
        \label{tab:error_tableBO_t2CN}
\end{table}

To achieve the higher order accuracy of the proposed LDG scheme, we perform the numerical results using the four-stage fourth-order explicit Runge-Kutta LDG scheme \eqref{RK4scheme} of 
\begin{align}\label{num:temp_1}
U_t = -\frac{1}{2}\left( U^2\right)_x + \mathcal{H}U_{xx}.
\end{align}
Our focus lies on verifying the performance of the scheme using a low storage explicit Runge-Kutta (LSERK4) of fourth-order time discretization \cite{carpenter1994fourth,dwivedi2024local,hesthaven2007nodal} of the form
\begin{align*}
    &y^{(0)} = u_h^n,\\
    &k^0 =0,\\
    &\text{for } j =1:5\\
    & \qquad \begin{cases}
        k^j = a_j k^{j-1} +\tau L_hy^{(j-1)},\\
        y^{(j)} = y^{(j-1)}+b_j k^{j},
    \end{cases}\\
    &u_h^{n+1} = y^{(5)},
\end{align*}
where the weighted coefficients $a_j$, $b_j$ and $c_j$ of the LSERK4 method are given in \cite{hesthaven2007nodal}.  The above defined iteration is equivalent to the classical fourth-order method \eqref{RK4scheme} \cite{hesthaven2007nodal}. LSERK is considerably more efficient and accurate than \eqref{RK4scheme} since it has the disadvantage that it requires four extra storage arrays.  To compute the approximate solution using the LDG scheme \eqref{LDGschemeBO} with LSERK time discretization,  one may refer to \cite{hesthaven2007nodal}. 

In numerical implementation with LSERK time discretization, we initially compare the approximate solution obtained from the LDG scheme \eqref{LDGschemeBO} with the exact solution \eqref{BOsolution} at time $t=10$, and further at the final time $t=20$. Figure \ref{fig:BOLDG} depicts the comparison of the solution at various times.
This validation confirms that the numerical scheme converges to the exact solution. The Table \ref{tab:error_tableBO_t1} presents the $L^2$-errors of the proposed scheme under polynomial degree up to three. It is observed that the errors are converging to zero with relatively coarser grids and optimal orders of convergence are obtained at time $T=10$ for each polynomial degree up to three and the discrete conserved quantities $\mathbf C_1^h$ and $\mathbf C_2^h$ are also preserved. Similar analysis has been carried out at time $T=20$ and details are presented in Table \ref{tab:error_tableBO_t2} and optimal rates are obtained there as well.
\begin{figure}
    \centering
    \includegraphics[width=0.8 \linewidth, height=8cm]{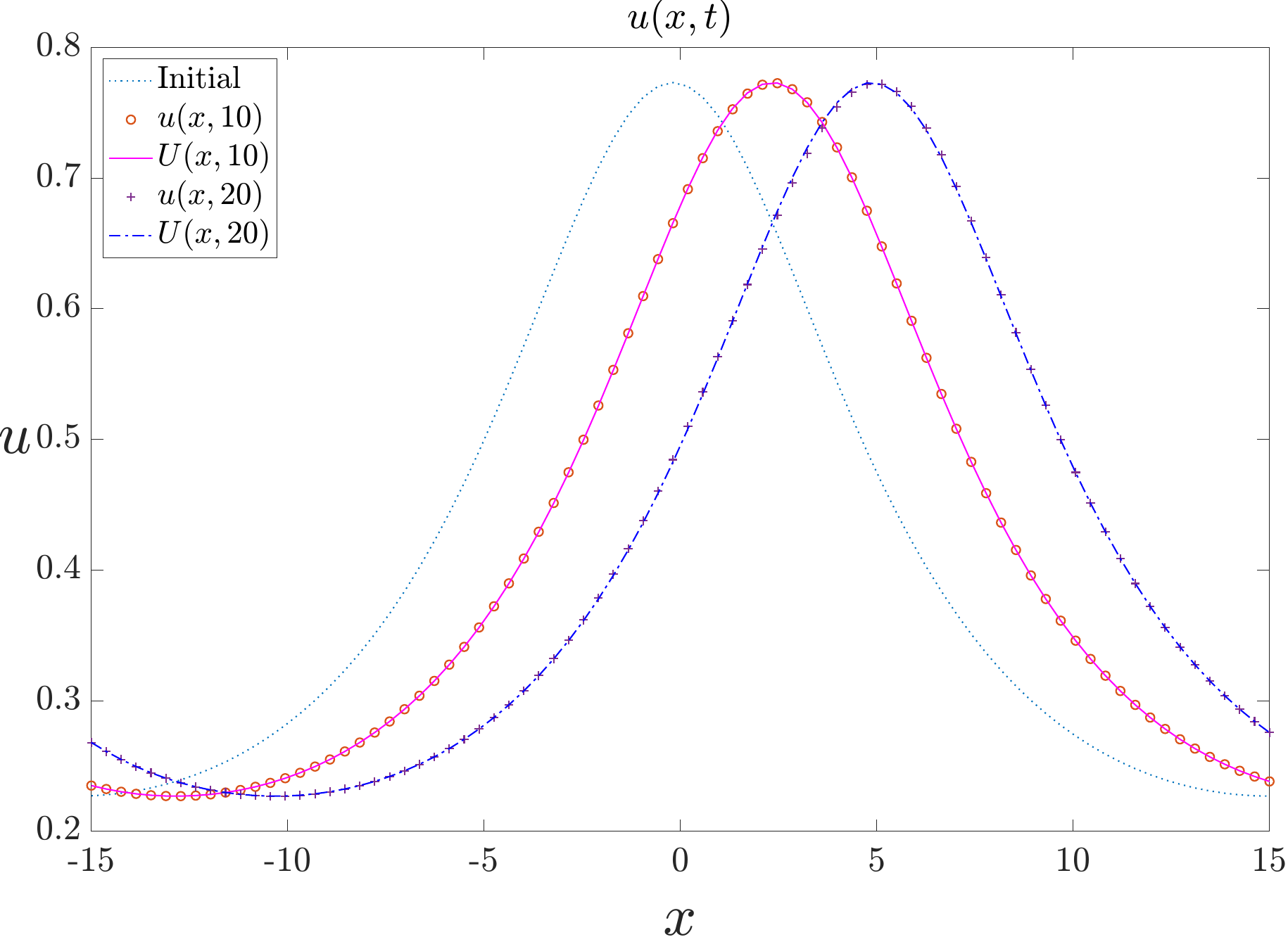}
    \caption{Exact $U(x,t)$  and approximate solution $u(x,t)$ computed by the RK-LDG scheme at $T=10$ and $T=20$ with $N=160$, $k=3$ and initial condition $U(x,0)$ of \eqref{BOeqn}.}
    \label{fig:BOLDG}
\end{figure}
\begin{table}[htbp]
        \centering
        \begin{tabular}{||c|c|c|c|c|c|c|c|c||}
            \hline
            $k$ & \multicolumn{2}{c|}{$k=1$}& \multicolumn{2}{c|}{$k=2$} & \multicolumn{4}{c||}{$k=3$} \\
            \hline
            \hline
          $N$  & $E$ & $R_E$  & $E$ & $R_E$ & $E$ & $R_E$ & $\mathbf C_1^h$ & $\mathbf C_2^h$ \\
            \hline
              40 & 2.91e-01 &  & 3.65e-02 &  & 8.69e-03& & 1.01 & 1.01 \\  
            &&  2.01&   &  2.76& & 4.37 &&\\  
             80 & 7.15e-02&& 5.37e-03 &  & 4.20e-04 & & 1.00 &1.00\\
              &&2.02&  & 2.93 & & 4.07 && \\
             160 & 1.77e-02&  & 7.029e-04 &  & 2.49e-05& & 1.00 & 1.00\\  
            &   & 2.00& & 2.98 & & 4.00&&\\  
             320 &4.40e-03& & 8.88e-05 &  & 1.56e-06 & & 1.00 & 1.00\\
            \hline
        \end{tabular}
        \caption{$L^2$-error and rate of convergence $R_E$ for the RK-LDG scheme at time $T=10$ taking $N$ elements and polynomial degree $k$ with normalized conserved quantities $\mathbf{C_1^h}$ and $\mathbf{C_2^h}$.}
        \label{tab:error_tableBO_t1}
 \end{table}
   \begin{table}[htbp]
        \centering
        \begin{tabular}{||c|c|c|c|c|c|c|c|c||}
            \hline
             $k$ & \multicolumn{2}{c|}{$k=1$}& \multicolumn{2}{c|}{$k=2$} & \multicolumn{4}{c||}{$k=3$} \\
            \hline
            \hline
          $N$  & $E$ & $R_E$  & $E$ & $R_E$ & $E$ & $R_E$ & $\mathbf C_1^h$ & $\mathbf C_2^h$\\
            \hline
              40  & 6.02e-01&& 6.015e-01 &  & 5.98e-02& & 0.99 &0.98 \\  
            & & 2.05&  & 2.65 & & 4.46 &&\\  
           80 & 1.45e-01&& 9.54e-02 &  &2.718e-03 & &1.00 &1.00\\
            & &2.02& & 2.90 & & 4.09 && \\
             160 &3.56e-02&  & 1.27e-02 &  & 1.59e-04& &1.00 &1.00\\  
            &  &2.01&  & 2.97 & & 4.02&&\\  
            320 &8.83e-03& & 1.62e-03 &  & 9.80e-06 & &1.00 &1.00\\
            \hline
        \end{tabular}
        \caption{$L^2$-error and rate of convergence $R_E$ for the RK-LDG scheme at time $T=20$ taking $N$ elements and polynomial degree $k$ with normalized conserved quantities $\mathbf{C_1^h}$ and $\mathbf{C_2^h}$.}
        \label{tab:error_tableBO_t2}
\end{table}

\subsection*{Example 2} We consider a two-soliton solution $U_{2}(x,t)$ to the Benjamin-Ono equation \eqref{BOeqn}, derived using the inverse scattering transform method. The expression is given by \cite{thomee1998numerical,galtung2018convergent}:
\begin{equation}
U_{2}(x,t) = \frac{4 c_1 c_2 \left( c_1 \lambda_1^2 + c_2 \lambda_2^2 + \frac{(c_1 + c_2)^3}{c_1 c_2 (c_1 - c_2)^2} \right)}{\left( c_1 c_2 \lambda_1 \lambda_2 - \frac{(c_1 + c_2)^2}{(c_1 - c_2)^2} \right)^2 + (c_1 \lambda_1 + c_2 \lambda_2)^2},
\end{equation}
where $\lambda_1$ and $\lambda_2$ are given by
\[
    \lambda_1 = x - c_1 t - d_1, \quad \lambda_2 = x - c_2 t - d_2,
\]
and we choose the parameters $ c_1 = 0.3, \, c_2 = 0.6, \, d_1 = -30, \, d_2 = -55.$ The LDG formulation involves evaluation of terms of the form $(\mathcal{H} \phi_i, \psi_i)$ for basis functions $\phi_i, \psi_i$ in each element $I_i$. These integrals are computed using the Gaussian quadrature. For the first Hilbert integral, we use a 7-point Gaussian quadrature in each element $I_i$. For the second Hilbert integral involving interaction terms, we use an 8-point Gaussian quadrature. This ensures high accuracy in evaluating the nonlocal dispersive contribution while retaining efficiency.

We set the initial data as $u_0(x) = U_{2}(x,0)$ and run the simulation until final time $T = 180$. We compare the numerical solution $u(x, T)$ with the exact two-soliton profile $U_{2}(x, T)$ using $L^2$-norm error metrics. The mesh consists of $N$ uniform elements and we test polynomial degrees $k = 1, 2, 3$. We employ the LSERK4 scheme to achieve high-order accuracy in time while maintaining computational efficiency.

Table \ref{tab:twosoliton_error_table} confirms that the proposed LDG scheme achieves the optimal convergence rates of order $k+1$ in the $L^2$-norm for polynomial degree $k = 1, 2, 3$, verifying the accuracy of the spatial discretization.
Figure \ref{fig:BOLDG2} shows the comparison of the exact and numerical two-soliton convergence at $T = 180$ for $k=3$ with $N=2560$. The LDG scheme successfully captures the nonlinear interaction and the propagation of both solitons with high accuracy. 

The numerical experiment confirms the accuracy and efficiency of the LDG scheme for solving the generalized Benjamin-Ono equation. By carefully selecting two-soliton parameters, we validated that the scheme accurately captures nonlinear wave interactions. Both Crank–Nicolson and LSERK4 time discretizations yield optimal rates of convergence. The high-order Gaussian quadrature used for evaluating the Hilbert transform further ensures the fidelity of the method.

\begin{table}[htbp]
    \centering
    \begin{tabular}{||c|c|c|c|c|c|c|c|c||}
        \hline
        $k$ & \multicolumn{2}{c|}{$k=1$} & \multicolumn{2}{c|}{$k=2$} & \multicolumn{4}{c||}{$k=3$} \\
        \hline
        \hline
        $N$ & $E$ & $R_E$ & $E$ & $R_E$ & $E$ & $R_E$ & $\mathbf{C_1^h}$ & $\mathbf{C_2^h}$ \\
        \hline
        640  & 5.82e-01 &      & 4.92e-02 &      & 3.86e-03  &      & 0.997 & 0.984 \\
            &          & 2.08 &           & 2.77 &           & 3.93 &       &       \\
         1280 & 1.36e-01 &      & 7.36e-03  &      & 2.47e-04  &      & 0.999 & 0.998 \\
            &          & 2.01 &           & 2.91 &           & 3.97 &       &       \\
        2560 & 3.38e-02 &      & 9.71e-04  &      & 1.57e-05  &      & 1.000 & 1.000 \\
            &          & 2.00 &           & 2.95 &           & 3.92 &       &       \\
        5120 & 8.44e-03 &      & 1.23e-04  &      & 1.03e-06  &      & 1.000 & 1.000 \\
        \hline
    \end{tabular}
    \caption{$L^2$-error and convergence rate $R_E$ for the LDG scheme with LSERK4 time integration at $T=180$ for the two-soliton solution, using $N$ elements and polynomial degree $k$. Normalized conserved quantities $\mathbf{C_1^h}$ and $\mathbf{C_2^h}$ confirm numerical mass and energy conservation.}
    \label{tab:twosoliton_error_table}
\end{table}
\begin{figure}
    \centering
    \includegraphics[width=0.8 \linewidth, height=8cm]{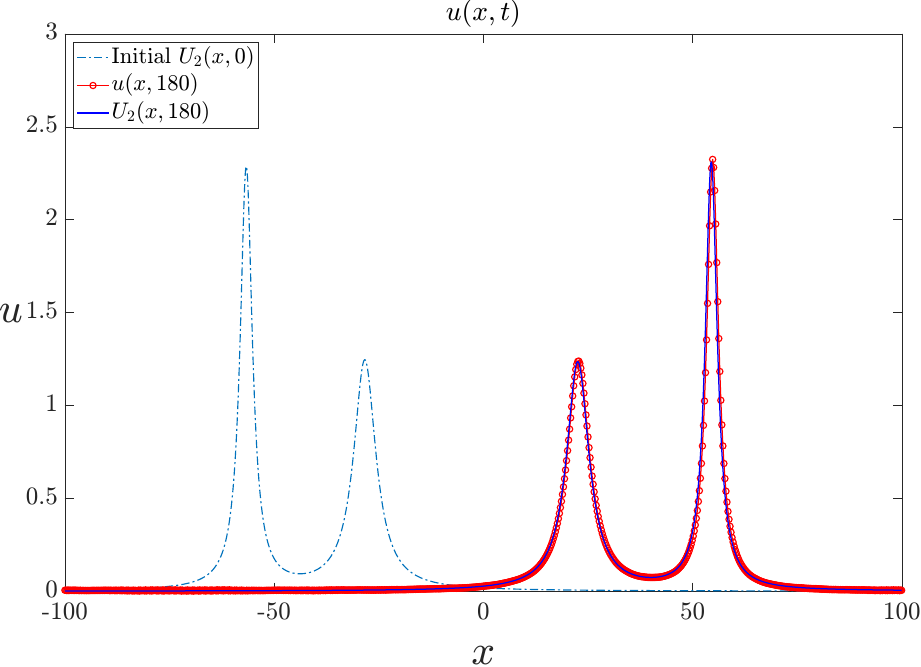}
    \caption{Exact $U_2(x,t)$  and approximate solution $u(x,t)$ computed by the RK-LDG scheme at $T=180$ with $N=2560$, $k=3$ and initial condition $U_2(x,0)$ of \eqref{BOeqn}.}
    \label{fig:BOLDG2}
\end{figure}


\section{Concluding remarks}\label{sec6}
We have designed a local discontinuous Galerkin method for the Benjamin-Ono equation with general nonlinear flux. We have shown that the semi-discrete scheme is stable and obtained a suboptimal order of convergence. The stability analysis is also carried out for the fully discrete Crank-Nicolson LDG scheme considering any general nonlinear flux function and strong stability of the fourth-order Runge-Kutta LDG scheme under certain assumptions on the flux function. The theoretical convergence analysis for the fully discrete scheme incorporating the fourth-order Runge-Kutta time marching scheme involving general nonlinear flux will be addressed in future work. In the numerical experiments, it is observed that the $L^2$-error is quite small even for the coarser grids and the optimal rates are obtained for various degrees of polynomials which demonstrate the efficiency and accuracy of the proposed scheme. In addition, the proposed fully discrete schemes preserve the conservative quantities such as mass and momentum.




\section*{Conflicts of interest} 
The authors declare that they have no known competing financial interests that could have appeared to influence the work reported in this paper. Furthermore, no data were used for the research described in the article.


\end{document}